\documentclass[12pt, a4paper]{amsart}

\calclayout

\usepackage{amssymb}
\usepackage{mathrsfs}
\usepackage{enumitem}
\usepackage{tensor}
\usepackage{tikz-cd}
\usepackage{hyperref}
\hypersetup{
  colorlinks=true,
  linkcolor={red!50!black},
  filecolor=magenta,
  urlcolor=cyan,
  citecolor={blue!80!black}
}
\usepackage{amsrefs}

\newcommand{\R}{\mathbb{R}}
\newcommand{\Z}{\mathbb{Z}}

\DeclareMathOperator{\dom}{dom}
\DeclareMathOperator{\Diff}{Diff}
\DeclareMathOperator{\Hol}{Hol}
\DeclareMathOperator{\pr}{pr}

\newcommand{\cplx}[2]{\Omega_{#1}^\bullet(#2)}

\newcommand{\cl}[1]{\mathcal{#1}}
\newcommand{\fk}[1]{\mathfrak{#1}}

\newcommand{\pdd}[2]{\frac{\partial #1}{\partial #2}}

\newcommand{\rra}{\rightrightarrows}
\newcommand{\fiber}[2]{\tensor[_{#1}]{\times}{_{#2}}}

\newcommand{\define}[1]{\emph{#1}}

\theoremstyle{plain}
\newtheorem{theorem}                 {Theorem}[section]
\newtheorem{proposition}   [theorem] {Proposition}
\newtheorem{lemma}         [theorem] {Lemma}
\newtheorem{corollary}     [theorem] {Corollary}

\newtheorem*{theorem*}     {Theorem}
\newtheorem*{question}     {Question}

\theoremstyle{definition}
\newtheorem{definition}    [theorem]{Definition}

\theoremstyle{remark}
\newtheorem{remark}        [theorem]{Remark}
\newtheorem{example}       [theorem]{Example}

\def\eor{\unskip\ \hglue0mm\hfill$\diamond$\smallskip\goodbreak}
\def\eoe{\unskip\ \hglue0mm\hfill$\between$\smallskip\goodbreak}
\def\eod{\unskip\ \hglue0mm\hfill$\diamond$\smallskip\goodbreak}

\numberwithin{equation}{section}

\newif\ifdebug

\debugfalse

\newcommand{\mute}[2] {\ifdebug {\textcolor{orange}{#1}} \fi}                          

\begin{document}
\title{Singular foliations through diffeology}
\author{David Miyamoto}
\address{University of Toronto}
\email{david.miyamoto@mail.utoronto.ca}
\date{\today}

\begin{abstract}
  A singular foliation is a partition of a manifold into leaves of perhaps varying dimension. Stefan and Sussmann carried out fundamental work on singular foliations in the 1970s. We survey their contributions, show how diffeological objects and ideas arise naturally in this setting, and highlight some consequences within diffeology. We then introduce a definition of transverse equivalence of singular foliations, following Molino's definition for regular foliations. We show that, whereas transverse equivalent singular foliations always have diffeologically diffeomorphic leaf spaces, the converse holds only for certain classes of singular foliations. We finish by showing that the basic cohomology of a singular foliation is invariant under transverse equivalence.
\end{abstract}

\maketitle

\section{Introduction}
\label{sec:introduction}

A singular foliation is a partition of a manifold into connected submanifolds, called leaves, of perhaps varying dimension, which fit together smoothly in an appropriate sense. Singular foliations lie beneath many structures one frequently encounters in geometry: the partition of a manifold into orbits of a connected smooth Lie group action, the partition of a Poisson manifold into symplectic leaves, and the partition of a manifold into the maximal integral submanifolds of an involutive distribution are all singular foliations. In the last example, the leaves of the resulting singular foliation all have the same dimension, and thus we obtain a regular foliation. A global version of the Frobenius theorem asserts that every regular foliation arises in this way.
\begin{theorem}[Global Frobenius]
  \label{thm:1}
  Let $M$ be a manifold. An involutive distribution $\Delta$ has integral submanifolds, and the partition of a manifold $M$ into the maximal integral submanifolds is a regular foliation $\cl{F}_\Delta$. Conversely, given a regular foliation $\cl{F}$, the associated distribution $T\cl{F}$ is involutive. The assignments
  \begin{equation*}
    \cl{F} \mapsto T\cl{F}, \quad \Delta \mapsto \cl{F}_\Delta
  \end{equation*}
  are inverses of each other.
\end{theorem}
We employ the adjective ``global'' because in his original paper \cite{Frobenius1877}, Frobenius worked locally, and did not consider foliations; in modern terms (as opposed to the language of Pfaffian systems used at the time), Frobenius proved that a regular distribution $\Delta$ has integral submanifolds if and only if it is involutive. Furthermore, Frobenius was not - and did not claim to be - the first to prove the result that now bears his name. He was building from the work of Jacobi \cite{Jacobi1827}, who gave sufficient integrability conditions in a special case, and of Clebsch \cite{Clebsch1866}, who expanded Jacobi's conditions the general case. Frobenius also singled out the contributions of Deahna \cite{Deahna1840}, who gave necessary and sufficient conditions for integrability (c.f. Remark \ref{rem:4}). For more on the history of Frobenius's theorem, see the treatments by Hawkins \cite{Hawkins2005} and Samelson \cite{Samelson2001}.  

For singular foliations the situation remained unclear for some time. Hermann \cite{Hermann1962} was the first to give some sufficient conditions for integrability in terms of Lie subalgebras of vector fields (stated in Example \ref{ex:3}). Shortly after, Nagano \cite{Nagano1966} established Frobenius's theorem holds in the analytic case. Matsuda \cite{Matsuda1968}, and Lobry \cite{Lobry1970} worked to extended the sufficient conditions of Hermann, but like Hermann did not propose integrability conditions on the singular distributions themselves. Then, in 1973-1974, Stefan \cite{Stefan1974} and Sussmann \cite{Sussmann1973} independently published results that, taken together, successfully generalize the global Frobenius theorem as stated above; the integrability condition on $\Delta$ is that it is invariant under the flow of any vector field tangent to $\Delta$.  In Section \ref{sec:singular-foliations} of this paper, we proceed carefully toward Stefan and Sussmann's results, highlighting the independent contributions of both authors. For a broader perspective on the history of integrability theorems for singular distributions, see Lavau's survey \cite{Lavau2018}.

Diffeology was introduced in the 1980s by Souriau, and was therefore not available to Stefan and Sussmann as they wrote in the early 1970s. However, Stefan's work in particular uses many tools that arise naturally in diffeology, and he draws conclusions that are best interpreted diffeologically. We emphasize these connections to diffeology throughout Section \ref{sec:singular-foliations}, but highlight here Corollary \ref{cor:5} and Proposition \ref{prop:2}, which imply that it is equivalent to consider singular foliations, and orbits of D-connected diffeological groups acting smoothly on $M$.

In Section \ref{sec:transv-equiv-sing}, we contrast two approaches to the transverse geometry of a singular foliation: one equips the leaf space $M/\cl{F}$ with the quotient diffeology, and the other defines a notion of transverse equivalence directly. We call this Molino transverse equivalence, after Molino, whose notion of transverse equivalence ({\cite[Definition 2.1]{Molino1988}}) for regular foliations this generalizes. By Proposition \ref{prop:6}, a Molino transverse equivalence always induces a diffeomorphism between the leaf spaces, but the converse does not hold in general; in Example \ref{ex:9}, we give two regular foliations that are not Molino transverse equivalent, but which have diffeomorphic leaf spaces. It is thus a substantive question whether there is a class of singular foliations for which the diffeology of its leaf space determines its Molino transverse equivalence class. We give one affirmative answer in Proposition \ref{prop:9} in terms of quasifold groupoids, which were introduced in \cite{Karshon2022}. This specializes to orbifolds.

Our Molino transverse equivalence is intimately related to Garmendia and Zambon's Hausdorff Morita equivalence of singular foliations, defined in \cite{Garmendia2019}. We note that Garmendia and Zambon begin with a substantially different notion of singular foliation, which we call an Androulidakis-Skandalis singular foliation, c.f.\ Definition \ref{def:9}. They therefore have a different, generally finer, notion of equivalence. Our Molino transverse equivalence may be viewed as an extension of Garmendia and Zambon's Hausdorff Morita equivalence to Stefan singular foliations (Definition \ref{def:3}).

Finally, in Section \ref{sec:basic-cohom-sing}, we apply our Molino transverse equivalence to the basic cohomology associated of a singular foliation. This cohomology was first introduced by Reinhart \cite{Reinhart1959} in the regular case. Our main result of this section, Corollary \ref{cor:4}, shows that the basic cohomology is invariant under Molino transverse equivalence. This is consistent with the expectation that the basic cohomology depends only on the transverse geometry of the singular foliation. As a consequence, we see by Proposition \ref{prop:11} that if the quotient map $\pi:M \to M/\cl{F}$ induces an isomorphism $\pi^*:H^\bullet(M/\cl{F}) \to H_b^\bullet(M, \cl{F})$, where $H^\bullet(M/\cl{F})$ is the cohomology associated to the diffeological differential forms on $M/\cl{F}$, the same is true for all singular foliations in the Molino transverse equivalence class of $(M, \cl{F})$. Whether $\pi$ always induces an isomorphism remains an open question.

We will assume familiarity with diffeology, and not review it here. For an excellent introduction, see the book by Iglesias-Zemmour \cite{Iglesias-Zemmour2013}. For a history of the development of singular foliations, see the survey by Lavau \cite{Lavau2018}. If the reader is unfamiliar with Lie groupoids, we have provided a very brief overview in Appendix \ref{sec:review-lie-groupoids}, but refer to the book by Moerdijk and Mrcun \cite{Moerdijk2003}, and the survey by Lerman \cite{Lerman2010}, for more details.

We will reiterate the structure of this article. In Section \ref{sec:singular-foliations}, we give a careful description of the definition of a singular foliation, and an overview of the Stefan-Sussmann theorem. We pay special attention to diffeological considerations. In Section \ref{sec:transv-equiv-sing}, we introduce Molino transverse equivalence of singular foliations, and relate this to the diffeology of the leaf space, and Morita equivalence of Lie groupoids. Finally, in Section \ref{sec:basic-cohom-sing}, we show that the basic cohomology of a singular foliation is invariant under Molino transverse equivalence. Appendix \ref{sec:review-lie-groupoids} establishes the notation and language we need when dealing with Lie groupoids and algebroids.

\subsection{Acknowledgements}
\label{sec:acknowledgements}

I am grateful to my advisor, Yael Karshon, for introducing me to singular foliations, and for describing the idea of Example \ref{ex:9}. I also thank Jordan Watts and Jean-Pierre Magnot for their invitation to speak at the AMS-EMS-SMF Joint International Meeting in Grenoble, and Jean-Pierre Magnot for the invitation to submit this article to the present conference proceedings. Finally, I thank the anonymous reviewer and Alfonso Garmendia for several helpful comments on earlier drafts.

\section{The Basics of Singular Foliations}
\label{sec:singular-foliations}

A singular foliation is a partition of a manifold\footnote{by ``manifold,'' we mean set equipped with a maximal smooth atlas (smooth structure) whose induced topology is Hausdorff and second-countable. When required, we may denote the smooth structure by a letter such as $\sigma$, and refer to the manifold as $(M, \sigma)$.} $M$ into connected submanifolds, called \emph{leaves}, fitting together smoothly in an appropriate sense. Before we give a definition of singular foliation, however, we must establish what we mean by ``submanifold.'' First candidates are immersed submanifolds.
\begin{definition}
  \label{def:1}
A subset $L$ of $M$, together with a manifold structure $\sigma_L$, is an \define{immersed submanifold} of $M$ if the inclusion $\iota:(L, \sigma_L) \hookrightarrow M$ is an injective immersion. \eod
\end{definition}
\begin{remark}
  \label{rem:1}
  The smooth structure $\sigma_L$ is part of the data of the immersed submanifold $(L, \sigma_L)$. There may be other smooth structures $\sigma_L'$ on $L$ for which $(L, \sigma_L')$ is immersed, but is not diffeomorphic to $(L, \sigma_L)$. For example, a figure-eight may be immersed in $\R^2$ in two different ways. Note the topology of $(L, \sigma_L)$ is generally finer than the subspace topology.
  \eor
\end{remark}
Diffeology provides other candidates for submanifolds, namely diffeological and weakly-embedded submanifolds.
\begin{definition}
  \label{def:2}
  A subset $L$ of $M$ is a \define{diffeological submanifold} of $M$ if its subspace diffeology is a manifold diffeology. Diffeological submanifolds always carry the manifold structure given by their subspace diffeology. We call such $L$ a \define{weakly-embedded submanifold} if the inclusion $i: L \hookrightarrow M$ is an immersion.
  \eod
\end{definition}
\begin{remark}
  \label{rem:2}
  Weakly-embedded submanifolds have unique smooth structures for which the inclusion is an immersion; they are unambiguously immersed submanifolds. The topology of a diffeological submanifold is finer than the subspace topology. Due to a theorem of Joris \cite{Joris1982}, the cusp $x^2=y^3$ is a diffeological submanifold of $\R^2$ that is not weakly-embedded. For a detailed discussion, see \cite{Karshon2022b}.
  \eor
\end{remark}
Stefan required the leaves of a singular foliation to be weakly-embedded in his definition from \cite{Stefan1974}.
\begin{definition}
  \label{def:3}
  A \define{(Stefan) singular foliation} of a smooth manifold $M$ is a partition $\cl{F}$ of $M$ into connected, weakly-embedded submanifolds $L$, called \define{leaves} such that about each $x \in M$, there exists a chart $\psi$ of $M$ for which
  \begin{enumerate}[label = (\alph*)]
  \item $\psi$ is a diffeomorphism $V \to U \times W$, where $U$ and $W$ are open neighbourhoods of the origin in $\R^k$ and $\R^{n-k}$, respectively, and $k$ is the dimension of the leaf through $x$; 
  \item $\psi(x) = (0,0)$;
    \item If $L$ is any leaf of $\cl{F}$, then $\psi(L) = U \times \ell$, for some $\ell \subseteq W$.
    \end{enumerate}
    \eod
\end{definition}
\begin{remark}
  \label{rem:3}
  If every leaf has the same dimension, we call the foliation \emph{regular}. In this case, Stefan's definition coincides with the usual chart-based definition of a foliation.
  \eor
 \end{remark}
 \begin{remark}
   In the next section, we will consider the assignment to each $x \in M$ the subspace $T_xL\subseteq T_xM$. This is an example of a singular distribution, and the leaves $L$ are its maximal integral submanifolds (c.f. Definitions \ref{def:6}, \ref{def:7}, and \ref{def:8}).
   \eor
 \end{remark}
 An attractive feature of this definition is that it is intrinsic to the manifold $M$ and the partition into leaves $L$ (whose smooth structures are determined by the smooth structure on $M$, see Remark \ref{rem:2}). This is the first intrinsic definition of singular foliation to appear in the literature. Prior to Stefan, singular foliations were known only through examples appearing, for instance, in control theory.
  
\begin{example}
  \label{ex:1}
  Given a Lie group $G$ acting smoothly on a manifold $M$, the connected components of the $G$-orbits assemble into a singular foliation. More generally, the connected components of the orbits of a Lie groupoid $\cl{G} \rra M$ form a singular foliation of $M$. Yet more generally, every Lie algebroid $A \implies M$ induces a singular foliation on $M$. We treat these examples later in Example \ref{ex:3}. See Appendix \ref{sec:review-lie-groupoids} for basic information on Lie groupoids and algebroids.
  \eoe
\end{example}

The condition that the leaves of $\cl{F}$ are weakly-embedded submanifolds seems at first unnecessarily restrictive. \emph{A priori}, we might gain generality by only requiring that the leaves are connected immersed submanifolds of $M$. Kubarski \cite{Kubarski1990} in 1990 showed that in fact, we gain no generality from this weaker assumption.
\begin{theorem}[Kubarski]
  \label{thm:2}
  Suppose $M$ is a smooth manifold, and $\cl{F}$ is a partition of $M$ such that
  \begin{itemize}
  \item Each $L \in \cl{F}$ comes equipped with a smooth structure $\sigma_L$ for which $(L, \sigma_L)$ is a connected immersed submanifold of $M$, and
  \item About each $x \in M$, there is a chart $\psi$ of $M$ satisfying (a) through (c) in Definition \ref{def:3}, where the $k$ in condition (a) is the dimension of the leaf $(L, \sigma_L)$ about $x$.
  \end{itemize}
  Then each $L$ is a weakly-embedded submanifold of $M$, and consequently, $\cl{F}$ is a Stefan singular foliation.
\end{theorem}
On the other hand, the notion of a diffeological submanifold provides another potential generalization to Stefan's definition. We show that again, we gain no generality under this weaker assumption.

\begin{proposition}
  \label{prop:1}
  Suppose $M$ is a smooth manifold, and $\cl{F}$ is a partition of $M$ into connected, diffeological submanifolds such that about each $x \in M$, there exists a chart $\psi$ of $M$ satisfying (a) through (c) in Definition \ref{def:3}. Then each $L \in \cl{F}$ is a weakly-embedded submanifold of $M$, and consequently, $\cl{F}$ is a Stefan singular foliation.
\end{proposition}
\begin{proof}
  It suffices to show the inclusion $\iota:L \hookrightarrow M$ is an immersion. Take $x \in L$, and fix a chart $\psi$ satisfying (a) through (c) of Definition \ref{def:3}. The intersection $L \cap V$ is an open subset of $L$.\footnote{recall the topology on $L$ is finer than the subspace topology.} Consider the diagram below (slightly abusing notation)
  \begin{equation*}
    \begin{tikzcd}
      L \cap V \ar[r, hook, "\iota"] & M \\
      U \times \{0\} \ar[u, "\psi^{-1}"] \ar[ur, "\psi^{-1}"'].
    \end{tikzcd}
  \end{equation*}
  Because the manifold structure on $L$ is given by the subset diffeology inherited from $M$, all maps in this diagram are smooth. The map on the right, $\psi^{-1}:U \times \{0\} \to M$, is an immersion, and by the chain rule, so is the map on the left, $\psi^{-1}:U \times \{0\} \to L \cap V$. Because both $U \times \{0\}$ and $L \cap V$ are manifolds of dimension $\dim L$, the map on the left, $\psi^{-1}:U \times \{0\} \to L \cap V$, is a diffeomorphism onto its image $L_0 := \psi^{-1}(U \times \{0\})$. Therefore it is a chart of $L$. In this chart, the inclusion satisfies
  \begin{equation*}
    \begin{tikzcd}
      L_0 \ar[r, hook, "\iota"] & V \ar[d, "\psi"] \\
      U \times \{0\} \ar[u, "\psi^{-1}"] \ar[r, hook] & U \times W.
    \end{tikzcd}
  \end{equation*}
  Therefore $\iota$ is locally an inclusion, hence is an immersion.
\end{proof}

\subsection{The Stefan-Sussmann Theorem}
Stefan's definition is ultimately difficult to verify, even when the proposed foliation should be regular. In the regular case, we instead identify foliations with their associated distributions, and apply the Frobenius theorem. Stefan \cite{Stefan1974} and Sussmann \cite{Sussmann1973} extended the Frobenius theorem simultaneously but independently in 1973-1974 to the singular setting, in what is now called the Stefan-Sussmann theorem. However, while similar, their contributions were not identical, and we will take care here to highlight their separate work.

Stefan worked with singular foliations generated by arrows. These are simply certain $1$-plots of $\Diff_{\text{loc}}(M)$, the space of diffeomorphisms between open subsets of $M$ equipped with its functional diffeology (for a description of this diffeology, see \cite{Iglesias-Zemmour2018}).
\begin{definition}
  \label{def:4}
  An \define{arrow} on a manifold $M$ is a $1$-plot $a:U \to \Diff_{\text{loc}}(M)$, such that $a(0) = \operatorname{id}$ and $\dom(a(t))\subseteq \dom(a(s))$, whenever $0 \leq |s| \leq |t|$. We allow for $a(t)$ to be the empty map.
  \eod
\end{definition}
Given a collection $\cl{A}$ of arrows, the union of their images, $\bigcup_{a \in \cl{A}} a(U)$ generates a pseudogroup, which we denote $\Psi\cl{A}$. We recall here the definition of pseudogroup.

\begin{definition}
  \label{def:5}
  A \define{pseudogroup} $P$ on a manifold $M$ is a set of diffeomorphisms between open subsets of $M$ that contains the identity, is closed under composition, inversion, and restriction to open subsets, and satisfies: if $f$ is a diffeomorphism between open subsets, and $f$ is locally in $P$, then $f$ is in $P$.
  \eod
\end{definition}

Also associated to $\cl{A}$ are two smooth singular distributions, by which we mean:
\begin{definition}
  \label{def:6}
  A \define{(smooth) singular distribution} on a manifold $M$ is a subset $\Delta \subseteq TM$ such that, for every $x \in M$, the set $\Delta_x := \{v \in \Delta \mid \pi(v) = x\}$ is a vector subspace of $T_xM$, and for every $v \in \Delta$, there is a locally defined vector field $X$ such that $X_{\pi(v)} = v$ and $X_y \in \Delta_y$ for all $y \in \dom X$.
  \eod
\end{definition}

Namely, set
\begin{align*}
  (\Delta\cl{A})_x &:= \left\{\frac{d}{dt}\Big|_{t=t_0} a(t)(y) \mid a(t)(y) = x\right\} \\
   \overline{(\Delta\cl{A})}_x &:= \{d\varphi_y v \mid \varphi \in \Psi\cl{A} \text{ and } v \in (\Delta\cl{A})_x\text{ and } \varphi(y) = x\}.
\end{align*}
Stefan {\cite[Theorem 1]{Stefan1974}} proved:
\begin{theorem}[Stefan]
  \label{thm:3}
  Given a collection of arrows $\cl{A}$ on $M$, the partition $\cl{F}$ of $M$ into the orbits of $\Psi\cl{A}$ is a singular foliation, and $T_xL = \overline{(\Delta\cl{A})}_x$ for each $x \in M$, where $L$ is the leaf about $x$.
\end{theorem}

\begin{example}
  \label{ex:2}
  Consider two vector fields on $\R^2$,
  \begin{equation*}
    X := \pdd{}{x} \quad Y:= \varphi(x)\pdd{}{y},
  \end{equation*}
  where $\varphi$ is a smooth, bounded, non-negative, increasing function which vanishes if and only if $x \leq 0$, and $\varphi(x) = 1$ if $x \geq 1$. Let $\cl{A}$ consist of the flows of $X$ and $Y$. Then $\Psi\cl{A}$ acts transitively on $\R^2$: it clearly acts transitively on $x > 0$; if $x \leq 0$, first flow along $X$ until we enter the region $x > 0$, flow along $\pm Y$ until we arrive at the required $y$-coordinate, and then flow along $-X$.  It follows that $\overline{(\Delta \cl{A})}_{(x,y)} = \R^2$ at all $(x,y)$, and the only orbit of $\Psi \cl{A}$ is $L = \R^2$.
  \eoe
\end{example}

In the setting of Stefan's Theorem \ref{thm:3}, the fact that for each $x \in M$, we have $\overline{(\Delta\cl{A})}_x = T_xL$ (for the leaf $L$ about $x$) means that the singular distribution $\overline{\Delta\cl{A}}$ is integrable, in the sense of the next definition below.

\begin{definition}
  \label{def:7}
  A singular distribution $\Delta$ is \define{integrable} if, at every $x \in M$, it admits an \define{integral submanifold}, which is an immersed submanifold $(L, \sigma_L)$ containing $x$ such that $T_y(L, \sigma_L) = \Delta_y$ for all $y \in L$.
  \eod
\end{definition}
Suppose now that $D$ is a collection of locally defined vector fields on $M$ inducing a singular distribution $\Delta$, and $\cl{A}_D$ is the collection of plots of the form $t \mapsto \Phi^X(t,\cdot)$, where $\Phi^X$ is the flow of an element of $D$. We have $\Delta\cl{A}_D = \Delta$. Stefan's theorem implies
\begin{corollary}\label{cor:1}
  The orbits of $\Psi\cl{A}_D$ are integral submanifolds of $\Delta$ if and only if $\Delta = \overline{\Delta}$.
\end{corollary}

This corollary characterizes all integral singular distributions $\Delta$ whose integral submanifolds are the orbits of $\Psi \cl{A}_D$, for some (hence any) collection of locally defined vector fields $D$ spanning $\Delta$. However, assuming only that $\Delta$ is integrable, it is not obvious from the above statements that $\Delta = \overline{\Delta}$. We do not know \emph{a priori} that the integral submanifolds of $\Delta$ are the orbits of $\Psi \cl{A}_D$. Sussmann {\cite[Theorem 4.2]{Sussmann1973}} showed that nevertheless, for an integrable singular distribution $\Delta$, we do in fact have $\Delta = \overline{\Delta}$. He achieves this as a consequence of his Orbit Theorem {\cite[Theorem 4.1]{Sussmann1973}), which we state for completeness.
\begin{definition}
  \label{def:8}
  An integral submanifold $(L, \sigma_L)$ of $\Delta$ is \define{maximal} if it is connected, and every connected integral submanifold $(S, \sigma_S)$ of $\Delta$ which intersects $L$ is an open subset of $L$ (in its manifold topology).\footnote{Equivalently, $(L, \sigma_L)$ if, at each point in $L$, it is maximal if it is maximal with respect to the inclusion.}
  \eod
\end{definition}
\begin{theorem}[Sussmann's orbit theorem]
  \label{thm:4}
  Given a singular distribution $\Delta$ and set of locally-defined vector fields $D$ as above, the orbits of $\Psi\cl{A}_D$ are maximal integral submanifolds of $\overline{\Delta}$.
\end{theorem}

Combining Stefan and Sussmann's results yields what is today called the Stefan-Sussmann theorem.
\begin{theorem}[Stefan-Sussmann]
  \label{thm:5}
  A singular distribution $\Delta$ is integrable if and only if $\Delta = \overline{\Delta}$, in which case its maximal integral submanifolds are leaves of a singular foliation $\cl{F}$ such that $T_xL = \Delta_x$, where $L$ is the leaf through $x$.
\end{theorem}
\begin{remark}
  \label{rem:4}
  Stefan called the condition $\Delta\cl{A} = \overline{\Delta\cl{A}}$ \emph{homogeneity} of the set of arrows $\cl{A}$. Sussmann worked with a set $D$ of partially-defined vector fields spanning a singular distribution $\Delta$, and used the term \emph{$D$-invariance} for the condition $\Delta = \overline{\Delta}$. We will adopt Stefan's term ``homogeniety.'' This condition is implicit in earlier work by Lobry \cite{Lobry1970}, c.f. \cite[page 53]{Lavau2018}.

  When $\dim \Delta_x$ is constant over $M$, homogeneity of $\Delta$ is equivalent to involutivity, and we recover the global Frobenius theorem \ref{thm:1} for regular foliations. In fact, Deahna \cite{Deahna1840}, who provided necessary and sufficient conditions for integrability of a regular distribution before Frobenius, gave homogeniety of $\Delta$, and not involutivity, as the sufficient condition, c.f. \cite[page 525]{Samelson2001}.

  We already remarked that only Sussmann proved maximality of the integral submanifolds of an integrable $\Delta$. Only Stefan proved the orbits of $\Psi\cl{A}$ were weakly-embedded submanifolds, and not merely immersed, and only Stefan showed the orbits (i.e. the maximal integral submanifolds) assemble into a singular foliation.
  \eor
\end{remark}

\begin{example}
  \label{ex:3}
  Let $A \implies M$ be a Lie algebroid, with anchor $\rho:A \to TM$. The image of the anchor, $\rho(A)$, is an integrable singular distribution. In this case, a maximal integral submanifold of $\Delta$ through $x \in M$ coincides with the set of all points reachable from $x$ via $A$-paths, which are paths $\gamma$ in $M$ for which there exists a path $\tilde{\gamma}$ in $A$ such that $\rho \circ \tilde{\gamma} = \dot \gamma$.

  By passing to its associated Lie algebroid, one finds that the connected components of the orbits of a Lie groupoid $\cl{G} \rra M$, and the connected components of the orbits of a Lie group $G$ acting smoothly on $M$, assemble into a singular foliation.

  Historically, integrability of $\rho(A)$ follows from a result of Hermann \cite{Hermann1962} that predates Stefan and Sussmann's work. Hermann showed that any \emph{locally finitely generated Lie sub-algebra}\footnote{meaning a Lie subalgebra $D$ of $\mathfrak{X}(M)$ such that, for every open $U\subseteq M$, there exist $X^1,\ldots, X^k \in D$ such that $D|_U \subseteq C^\infty(U)\text{span}(X^1|_U, \ldots, X^k|_U)$.} of $\mathfrak{X}(M)$ induces an integrable singular distribution, and we can show that $\rho(A)$ satisfies this hypothesis. That the integral submanifolds are maximal, and assemble into a singular foliation, still requires Stefan's contribution to the Stefan-Sussmann theorem.
  \eoe
\end{example}

As a corollary, we can completely describe singular foliations by their associated singular distributions. This generalizes the Frobenius theorem as stated in the Introduction.
\begin{corollary}\label{cor:2}
  Given a singular foliation $\cl{F}$, the collection $T\cl{F} := \bigcup_{x \in M} T_xL$ is an integrable singular distribution. Given an integrable singular distribution $\Delta$, the partition of $M$ into maximal integral submanifolds is a singular foliation, denoted $\cl{F}_\Delta$. The assignments
  \begin{equation*}
    \cl{F} \mapsto T\cl{F}, \quad \Delta \mapsto \cl{F}_\Delta
  \end{equation*}
  are inverses of each other.
\end{corollary}
\begin{remark}
  \label{rem:5}
  We make two comparisons to the regular case. First, homogeneity (see Remark \ref{rem:4}) implies involutivity, but for non-regular singular distributions, homogeneity is stronger than involutivity. Indeed, the singular distribution spanned by $X$ and $Y$ in Example \ref{ex:2} is involutive, but it is not integrable. For regular distributions, involutivity and homogeneity are equivalent.

  Second, the maximal integral submanifolds of an integrable regular distribution may also be described as equivalence classes of the relation: $x \sim y$ if there is a  path $\gamma$ from $x$ to $y$ tangent to $\Delta$, i.e.\ $\dot{\gamma}(t) \in \Delta_{\gamma(t)}$ for all $t$. In general, this is not the case for non-regular singular distributions. For example, consider the singular distribution spanned by the Euler vector field $X := x\pdd{}{x} + y \pdd{}{y}$ on $\R^2$. This has maximal integral submanifolds the origin $\{(0,0)\}$, and the rays from the origin. However, every two points can be joined by a path tangent to $\Delta$.
  \eor
\end{remark}

\subsection{Stefan's results and diffeology}
\label{sec:stef-results-diff}

While Stefan's Theorem \ref{thm:3} is the most cited part of his paper \cite{Stefan1974}, in the same work he also proved a result about orbits of subgroups of $\Diff(M)$, which has applications to diffeology. In this section, the $D$-\emph{topology}, and in particular $D$-\emph{connectedness} of a diffeological space $X$ is central.\footnote{Here the ``D'' stands for ``Diffeology.'' In this context, it does not represent some family of vector fields.} The $D$-topology on $X$ is the finest topology for which all the plots are continuous, and its $D$-connected components coincide with its (smooth) path-connected components. See Chapters 2 and 5 of \cite{Iglesias-Zemmour2013} for details.

Stefan proved {\cite[Theorem 3]{Stefan1974}}:
\begin{corollary}
  \label{cor:5}
  Let $G$ be a D-connected diffeological subgroup of $\Diff(M)$. Then the $G$-orbits assemble into a singular foliation of $M$.
\end{corollary}
\begin{proof}
  Let $\cl{A}$ consist of all plots (paths) $a:\R \to G$ with $a(0) = \operatorname{id}$. Then $\cl{A}$ is a set of arrows. Because $G$ is D-connected, it is path-connected ({\cite[Article 5.7]{Iglesias-Zemmour2013}}). Thus $\bigcup_{a \in \cl{A}} a(\R) = G$, and $\Psi\cl{A}$ is the pseudogroup generated by $G$, so the result follows from Theorem \ref{thm:3}.
\end{proof}
Conversely, we have:
\begin{proposition}\label{prop:2}
  Every singular foliation of $M$ is a partition of $M$ into orbits of some $D$-connected subgroup of $\Diff(M)$. 
\end{proposition}
\begin{proof}
  Fix a singular foliation $\cl{F}$. For each $x \in M$, take a collection $\{X^x\}$ of compactly supported vector fields whose values at $x$ span $T_x\cl{F}$. This is possible because $T\cl{F}$ is smooth (Definition \ref{def:6}, Corollary \ref{cor:2}), and $M$ supports bump functions. Let $\cl{A}$ denote the collection of all plots of the form $t \mapsto \Phi^{X^x}(t,\cdot)$. Then $\bigcup_{a \in \cl{A}}a(\R)$ generates some D-connected subgroup of $\Diff(M)$, and its orbits coincide with the orbits of $\Psi\cl{A}$, which are precisely the leaves of $\cl{F}$ by the Stefan-Sussmann Theorem \ref{thm:5}.
\end{proof}
\begin{remark}
  \label{rem:6}
  In the proof above, we required a collection of vector fields which spanned $T\cl{F}$. We were not concerned with how many vector fields were necessary. In 2008 and 2012, respectively, Sussmann \cite{Sussmann2008} and Drager, Lee, Park, and Richardson \cite{Drager2012} independently showed that only finitely many are required. Precisely, they showed that given a singular distribution $\Delta$, there exist finitely many vector fields $X^1,\ldots, X^n$ such that $\Delta_x = \text{span}\{X^1_x,\ldots, X^n_x\}$ for every $x \in M$. Drager et al. provided some bounds for $n$.  Note, however, that the collection $X^1, \ldots, X^n$ does not necessarily generate the sheaf of sections of $\Delta$. In fact, Drager et. al. \cite[Section 5]{Drager2012} give an example of a singular distribution over $\R$ whose sheaf of sections is not finitely generated.
  \eor
\end{remark}
As an immediate of Proposition \ref{prop:2}, we have:
\begin{corollary}
  For a singular foliation $\cl{F}$, the group of diffeomorphisms of $M$ which fix the leaves of $\cl{F}$ acts transitively on each leaf.
\end{corollary}
It follows that, for a manifold $M$, it is equivalent to consider singular foliations, to consider integrable distributions, and to consider orbits of D-connected subgroups of $\Diff(M)$. Here is one consequence of this fact.

\begin{proposition}
  \label{prop:3}
  The partition given by a singular foliation $\cl{F}$ of $M$ satisfies the frontier condition: if $S$ and $L$ are leaves, and $S \cap \overline{L} \neq \emptyset$, then $S \subseteq \overline{L}$ (here the closure is taken relative to $M$).
\end{proposition}
\begin{proof}
  Fix a D-connected subgroup $G$ of $\Diff(M)$ whose orbits are the leaves of $\cl{F}$. It suffices to show $\overline{L}$ is $G$-invariant. Suppose $x \in \overline{L}$, and let $(x_n)$ be a sequence of points in $L$ which converge (in $M$) to $x$. Then for any $g \in G$, we have $g(x_n) \to g(x)$, and all the $g(x_n)$ are in $L$, hence $g(x) \in \overline{L}$. Since $g$ was arbitrary, this completes the proof. 
\end{proof}

\subsection{Androulidakis-Skandalis singular foliations}
\label{sec:andr-skand-sing}
We view singular foliations as partitions of a manifold into leaves, and have seen that by the Stefan-Sussmann theorem, this is equivalent to considering integrable singular distributions. Some authors take a different approach, and view singular foliations as choices of submodules of vector fields. We use the definition given by Androulidakis and Skandalis \cite{Androulidakis2009}.
\begin{definition}
  \label{def:9}
  An \emph{Androulidakis-Skandalis singular foliation} is a $C^\infty(M)$-submodule $\mathscr{D}$ of $\fk{X}_c(M)$ that is locally finitely generated\footnote{this means that, about every $x \in M$, there is a neighbourhood $U$ and $Y^1,\ldots, Y^k \in \mathscr{D}|_U$ such that $\mathscr{D}|_U = C_c^\infty(U)\text{span}(Y^1,\ldots, Y^k)$, where $\mathscr{D}|_U$ is the $C^\infty(U)$-submodule of $\fk{X}_c(U)$ generated by all the $Y = fX|_U$, where $f \in C_c^\infty(U)$ and $X \in \mathscr{D}$.}  and involutive.
  \eod
\end{definition}

\begin{remark}
  \label{rem:7}
  As noted, for example, by Garmendia and Zambon \cite[Remark 1.8]{Garmendia2019}, and by Wang \cite[Remark 2.1.13]{Wang2017}, it is equivalent to define an Androulidakis-Skandalis singular foliation as an involutive, locally finitely generated subsheaf of the sheaf of vector fields on $M$. This approach could allow a definition of singular foliations on possibly non-Hausdorff manifolds, where the submodule and sheaf theoretic notions no longer coincide, c.f. Remark \ref{rem:8}.
  \eor
\end{remark}

The Stefan-Sussmann Theorem \ref{thm:5} implies that the singular distribution associated to an Androulidakis-Skandalis singular foliation $\mathscr{D}$ is integrable (or, one can apply an earlier result of Hermann, c.f. Example \ref{ex:3}). Therefore $\mathscr{D}$ induces a partition of a manifold $M$ into leaves, and this partition is a Stefan singular foliation. However, many different choices of $\mathscr{D}$ may induce the same singular foliation, as the example below shows.

\begin{example}
  \label{ex:4}
  Let $\mathscr{D}_k$ be generated by $x^k \pdd{}{x}$. Each $\mathscr{D}_k$ induces the Stefan singular foliation of $\R$ with three leaves: $x < 0$, $\{0\}$, and $x > 0$. However, no two of the modules $\mathscr{D}_k$ are isomorphic.
  \eoe
\end{example}

Every Lie algebroid, hence every Lie groupoid and action of a Lie group on a manifold, induces an Androulidakis-Skandalis singular foliation. Conversely, Androulidakis and Zambon {\cite[Proposition 1.3]{Androulidakis2013}} have described an Androulidakis-Skandalis singular foliation $\mathscr{D}_{\text{counter}}$ of $\R^2$ that is not induced by any Lie algebroid. We do not give the details here, but note that the underlying Stefan singular foliation is the partition of $\R^2$ into point leaves $\{(k,0)\}$ for natural $k$, and the complement $\R^2 \smallsetminus \bigcup_{k \geq 1} \{(k,0)\}$. This underlying Stefan singular foliation is induced by a smooth Lie group action on $\R^2$, namely that generated by vector fields $f(x,y)\pdd{}{x}$ and $f(x,y)\pdd{}{y}$, where $f$ is a bounded, non-negative function that vanishes precisely on the point leaves $\{(k,0)\}$. But this Lie group action does not induce the Androulidaks-Skandalis singular foliation $\mathscr{D}_{\text{counter}}$.

At time of writing, the following question is open:
\begin{question}
  Is every Stefan singular foliation induced by some Androulidakis-Skandalis singular foliation? 
\end{question}
As we see in Example \ref{ex:4} above, even if the answer is affirmative, the choice of Androulidakis-Skandalis singular foliation is not unique. It is also not natural, in general: consider the singular foliation $\cl{F}$ of $\R$ whose leaves are $\{x\}$ with $x \leq 0$, and $x > 0$. By {\cite[Proposition 5.3]{Drager2012}}, the space of all vector fields tangent to $\cl{F}$ is not locally finitely generated, and it is not clear how to choose a ``best'' vector field tangent to $\cl{F}$ to use as a generator of an Androulidakis-Skandalis singular foliation. 

\section{Transverse Equivalence of Singular Foliations}
\label{sec:transv-equiv-sing}

Given a singular foliation $\cl{F}$ of $M$, the leaf space $M/\cl{F}$ is naturally a diffeological space. For convenience, we recall the plots of $M/\cl{F}$.

\begin{definition}
  A map $p:U \to M/\cl{F}$, where $U$ is an open subset of Cartesian space, is a \define{plot} of the quotient diffeology of $M/\cl{F}$ if, about each $r \in U$, there is an open neighbourhood $V$ or $r$, and a smooth map $q:V \to M$, such that $p|_V = \pi \circ q$, where $\pi$ is the quotient map.
  \eod
\end{definition}
The D-topology coincides with the quotient topology \cite[Article 2.12]{Iglesias-Zemmour2013}. However, whereas the quotient topology on $M/\cl{F}$ may be trivial, or lose information about the singular foliation $\cl{F}$, its diffeology is often richer, as the following examples show.

\begin{example}
  \label{ex:5}
  Let $\pi:\R^2 \to T^2 := \R^2/\Z^2$ denote the quotient map for the $\Z^2$ action on $\R^2$. For irrational $\alpha$, let $S_\alpha$ be the line in $\R^2$ with slope $\alpha$. The conjugacy classes of the subgroup $\pi(S_\alpha)$ of $T^2$ assemble into a regular foliation of $T^2$. We call the leaf space $T_\alpha := T^2/\pi(S_\alpha)$ an \emph{irrational torus.}

  The quotient topology on $T_\alpha$ is always trivial, but the quotient diffeology is not: Donato and Iglesias-Zemmour \cite{Donato1985} proved that $T_\alpha \cong T_\beta$ if and only if $\beta = \frac{a+b\alpha}{c+d\alpha}$, where $a,b,c,d \in \mathbb{Z}$ and $ad-bc = \pm 1$.
  \eoe
\end{example}

\begin{example}
  \label{ex:6}
  Consider the action of the orthogonal group $O(n)$ on $\R^n$. The leaves of the induced singular foliation are the orbits of the action, namely the origin $\{0\}$ and the concentric spheres. Topologically, the quotient spaces $\R^n/O(n)$ are all homeomorphic to the half-line $[0,\infty)$ with its subspace topology. But by {\cite[Exercise 50,51]{Iglesias-Zemmour2013}} they are all diffeologically distinct, and none are diffeologically diffeomorphic to $[0,\infty)$ with its subset diffeology. For $n = 1$, we get a diffeological orbifold  $\R/\Z_2$.
  \eoe
\end{example}

We may, therefore, propose the diffeological space $M/\cl{F}$ as a model for the transverse geometry of $\cl{F}$. In the next sections, we compare this notion with another model.

\subsection{Molino Transverse Equivalence}
\label{sec:molino-transv-equiv}
In {\cite[Definition 2.1]{Molino1988}}, Molino defines a notion of transverse equivalence for regular foliations. In \cite{Garmendia2019}, Garmendia and Zambon extended this notion to Androulidakis-Skandalis singular foliations, and our definition below is similar to theirs and extends Molino's. However, due to the differences between Androulidakis-Skandalis and Stefan singular foliations outlined in Section \ref{sec:andr-skand-sing}, our definition does not coincide with Garmendia and Zambon's.

We need two technical Lemmas first.
\begin{lemma}
  \label{lem:1}
   Suppose $p:M \to N$ is a surjective submersion with connected fibers, and assume $N$ is connected. Then $M$ is connected.
\end{lemma}

\begin{proof}
  Take any continuous function $g:M \to \{0,1\}$.  Because the fibers of $p$ are connected, $g$ is constant on the fibers. Therefore, as $p$ is a submersion, there is a smooth function $h:N \to \{0,1\}$ such that $h \circ p = g$. Since $N$ is connected, $h(N)$ is a single point; since $p$ is surjective, $g(M) = h(N)$ is a single point. Therefore every continuous function $g:M \to \{0,1\}$ is constant, and we conclude that $M$ is connected.
\end{proof}

\begin{lemma}
  \label{lem:2}
    Suppose $p:M \to N$ is a submersion. For every locally defined vector field $Y$ on $N$, and every $v \in T_xM$ such that $dp_x(v) = Y_{p(x)}$, there is a locally defined vector field $X$ on $M$ such that $X_x = v$ and $X$ is $p$-related to $Y$.
  \end{lemma}

  \begin{proof}
    Because $p$ is a submersion, without loss of generality we may assume that $p:\R^m \to \R^n$ is the projection of the first $n$ coordinates. For $v = (v^1,\ldots, v^m) \in T_xM$, define
    \begin{equation*}
      X_{x'}:= (Y_{p(x')}, v^{n+1},\ldots, v^m).
    \end{equation*}
    Then $X$ is $p$-related to $Y$, because $dp_{x'}$ is the projection of the first $n$ coordinates, and the condition $dp_x(v) = Y_{p(x)}$ writes
    \begin{equation*}
      (v^1,\ldots, v^n) = Y_{p(x)},
    \end{equation*}
    so $X_x = v$.
  \end{proof}

  \begin{proposition}
    \label{prop:4}
    Let $p:M \to N$ be a surjective submersion with connected fibers. If $\Delta$ is an integrable singular distribution on $N$, then $(dp)^{-1}(\Delta)$ is an integrable singular distribution on $M$. If $L$ is a leaf of the singular foliation of $N$ induced by $\Delta$, then $p^{-1}(L)$ is a leaf of the singular foliation of $M$ induced by $(dp)^{-1}(\Delta)$.
  \end{proposition}
  \begin{proof}
    To see that $(dp)^{-1}(\Delta)$ is a smooth singular distribution, take $v \in ((dp)^{-1}(\Delta))_x$. Then $dp_x(v) \in \Delta_{p(x)}$, and since $\Delta$ is smooth, there is some partially defined vector field $Y$ of $N$ tangent to $\Delta$ with $dp_x(v) = Y_{p(x)}$. By Lemma \ref{lem:2}, we get a partially defined vector field $X$ of $M$ such that $X_x = v$, and $dp(X) = Y$. This vector field is tangent to $(dp)^{-1}(\Delta)$, and passes through $v$. Hence $(dp)^{-1}(\Delta)$ is smooth.

    To see that $(dp)^{-1}(\Delta)$ is integrable, let $x \in M$ and let $L$ be the maximal integral submanifold of $\Delta$ through $p(x)$. Because $p$ is a surjective submersion, $p^{-1}(L)$ is a weakly-embedded submanifold of $M$, and for each $x' \in p^{-1}(L)$,
    \begin{equation*}
      T_{x'}(p^{-1}(L)) = (dp_{x'})^{-1}(T_{p(x')}L) = ((dp)^{-1}(\Delta))_{x'}.
    \end{equation*}
    By the Stefan-Sussmann Theorem \ref{thm:5}, the maximal integral submanifolds of $(dp)^{-1}(\Delta)$ assemble into a singular foliation of $M$. But the $p^{-1}(L)$ already partition $M$ into connected (by Lemma \ref{lem:1}) integral submanifolds of $(dp)^{-1}(\Delta)$. Corollary \ref{cor:2} lets us conclude that the $p^{-1}(L)$ are the maximal integral submanifolds of $(dp)^{-1}(\Delta)$.
  \end{proof}
In light of Proposition \ref{prop:4}, we can use the following notation.
  \begin{definition}
    \label{def:10}
    Given a surjective submersion with connected fibers $p:M \to N$, and an integrable singular distribution $\Delta$ on $N$ with associated singular foliation $\cl{F}$, let $p^{-1}\Delta := (dp)^{-1}(\Delta)$, and let $p^{-1}\cl{F}$ denote the singular foliation associated to $p^{-1}\Delta$; its leaves are the sets $p^{-1}(L)$, where the $L$ are leaves of $\cl{F}$. We call $p^{-1}\Delta$ and $p^{-1}\cl{F}$ the \define{pullbacks} of $\Delta$ and $\cl{F}$.
    \eod
  \end{definition}

  Using the pullback foliation, we now define Molino transverse equivalence of singular foliations.

  \begin{definition}
    \label{def:11}
    Two singular foliations, $(N_0,\cl{F}_0)$ and $(N_1,\cl{F}_1)$, are \define{Molino transverse equivalent} if there exists a singular foliation $(M, \cl{F})$ and surjective submersions with connected fibers $p_i:M \to N_i$ such that $p_i^{-1}(\cl{F}_i) = \cl{F}$. We will write $\cl{F}_0 \cong \cl{F}_1$.
    \eod
  \end{definition}

Naturally, we must show:
  \begin{proposition}
    \label{prop:5}
  Molino transverse equivalence is an equivalence relation on singular foliations.    
\end{proposition}
\begin{proof}
  Reflexivity is witnessed by the identity. Symmetry is clear because we can reverse the roles of the $N_i$. For transitivity, build the following diagram:
  \begin{equation*}
      \begin{tikzcd}
      & & M \tensor[_{p_1}]{\times}{_{p_1'}} M' \ar[dl, "\pr_1"] \ar[dr, "\pr_2"'] & & \\
      & (M, \cl F) \ar[dl, "p_0"] \ar[dr, "p_1"'] & & (M', \cl F') \ar[dl, "p_1'"] \ar[dr, "p_2'"']& \\
      (N_0, \cl F_0) & & (N_1, \cl F_1) & & (N_2, \cl F_2).
    \end{tikzcd}
  \end{equation*}
The bottom two rows indicate the assumed Molino transverse equivalences. Denote the fiber product at the top by $M''$. This is a manifold because $p_1$ is a submersion. The projections are surjective submersions because the $p_i$ and $p_i'$ are surjective submersions, thus so too are the compositions $p_0 \circ \pr_1$ and $p_2' \circ \pr_2$. We claim that, moreover, these have connected fibers, and pull back $\cl{F}_0$ and $\cl{F}_2$, respectively, to the same singular foliation of $M''$.

  First, in light of Lemma \ref{lem:1}, it suffices to show, without loss of generality, that $\pr_1$ has connected fibers. Let $x \in M$, and take $(x,z)$ and $(x,z')$ in $M''$. Then $z$ and $z'$ are in the same fiber of $p_1'$, which is connected, hence there is a path $\gamma'$ joining them in this fiber. Then $(x,\gamma'(t))$ is a path in the fiber of $\pr_1$ joining $(x,z)$ and $(x,z')$, as required.

  The second assertion follows from the fact that any leaf $L$ of $\cl{F}$ has the form $p_0^{-1}(L_0) = p_1^{-1}(L_1)$, and any leaf $L'$ of $\cl{F}'$ has the form $(p_1')^{-1}(L_1') = (p_2')^{-1}(L_2')$, and from using the usual property of pullbacks that $(p_0 \circ \pr_1)^{-1}\cl{F}_0 = \pr_1^{-1}p_0^{-1}\cl{F}_0$, etc. This proves that $(M'', (p_0 \circ \pr_1)^{-1}(\cl{F}_0))$, with the surjective submersions $p_0 \circ \pr_1$ and $p_2' \circ \pr_2$, gives a Molino transverse equivalence between $\cl{F}_0$ and $\cl{F}_2$.
\end{proof}

We now relate Molino transverse equivalence and diffeology. One direction is straightforward: a Molino transverse equivalence induces a diffeological diffeomorphism of the leaf spaces.
\begin{proposition}
\label{prop:6}
Suppose $p:M \to N$ is a surjective submersion with connected fibers, and let $\cl{F}$ be a singular foliation of $N$. Then the map
\begin{equation*}
  \varphi:N/\cl{F} \to M/p^{-1}\cl{F}, \quad L \mapsto p^{-1}(L)
\end{equation*}
is a diffeological diffeomorphism. Consequently, if $p_i:M \to N_i$ witness a Molino transverse equivalence between $\cl{F}_i$, then the map
\begin{equation*}
  N_0/\cl{F}_0 \to N_1/\cl{F}_1, \quad L \mapsto p_1(p_0^{-1}(L))
\end{equation*}
is a diffeological diffeomorphism.
\end{proposition}
\begin{proof}
  The map $\varphi$ is well-defined by definition of $p^{-1}\cl{F}$, and its inverse if $p^{-1}(L) \mapsto L$. Both $\varphi$ and $\varphi^{-1}$ fit in the following diagram:
  \begin{equation*}
    \begin{tikzcd}
      M \ar[r, "p"] \ar[d, "\pi_1"] & N \ar[d, "\pi_2"]\\
      M/p^{-1}(\cl{F}) \ar[r, bend left, "\varphi^{-1}"]  & N/\cl{F} \ar[l, bend left, "\varphi"].
    \end{tikzcd}
  \end{equation*}
  Because $\pi_1$, and $\pi_2 \circ p$, are local subductions, so are $\varphi$ and $\varphi^{-1}$.
\end{proof}

\begin{example}
  \label{ex:7}
  Consider the singular foliations induced by the action of $O(n)$ on $\R^n$, from Example \ref{ex:6}. Because the $\R^n/O(n)$ are not diffeomorphic for different $n$, the singular foliations of $\R^n$ are not Molino transversely equivalent.
  \eoe
\end{example}

The converse is more subtle. Indeed, there exist two regular foliations, $(N_i,\cl{F}_i)$, with diffeomorphic leaf spaces, but which are not Molino transverse equivalent. We will see this by detouring through Morita equivalence of Lie groupoids.

\subsection{Molino versus Morita equivalence}
\label{sec:molino-versus-morita}

We saw in Example \ref{ex:1} and Example \ref{ex:3} that every Lie groupoid induces a singular foliation on its base manifold. For Lie groupoids, there is a long established notion of transverse, or weak, equivalence, called \emph{Morita equivalence}. In this subsection, we compare the notions of Molino and Morita equivalence. See Appendix \ref{sec:review-lie-groupoids} for the relevant language and constructions from the theory of Lie groupoids, which we use freely here.

\begin{proposition}
  \label{prop:7}
  Let $\cl{G} \rra M$ and $\cl{H} \rra N$ be two source-connected Lie groupoids with Hausdorff arrow spaces. If $\cl{G}$ and $\cl{H}$ are Morita equivalent, then the singular foliations $\cl{F}_{\cl{G}}$ and $\cl{F}_{\cl{H}}$ on $M$ and $N$ are Molino transverse equivalent.
\end{proposition}
\begin{proof}
  Fix an invertible bibundle $P:\cl{G} \to \cl{H}$ (see Definition \ref{def:15} for details)
\begin{equation*}
  \begin{tikzcd}
      \cl{G} \circlearrowright &  P  \ar[dl, "a"'] \ar[dr, "a'"] & \circlearrowleft \cl{H} \\
      M & &N.
    \end{tikzcd}
\end{equation*}
We claim that this exhibits a Molino transverse equivalence. First, since $\cl{G}$ and $\cl{H}$ are Hausdorff, by Lemma \ref{lem:3}, we may choose $P$ to be Hausdorff. Both $a$ and $a'$ are surjective submersions by assumption. Since $\cl{G}$ acts freely and transitively on the fibers of $a'$, the fiber of $a'$ over $x_0'$ is diffeomorphic to $s^{-1}(x_0') \cdot p_0$ for any fixed $p_0$ in the fiber, and this is connected because $\cl{G}$ is source-connected. Similarly, the fibers of $a$ are connected.

All that remains is to show the leaves of $a^{-1}\cl{F}_{\cl{G}}$ coincide with the leaves of $(a')^{-1}(\cl{F}_{\cl{H}})$. Let $\cl{O}$ be the orbit through a fixed $x_0$ in $M$. Choose $p_0 \in P$ with $a(p_0) = x_0$, and let $\cl{O}'$ be the orbit of $a'(p_0)$ in $N$. We claim $a^{-1}(\cl{O}) = (a')^{-1}(\cl{O}')$.

Letting $a(p) \in \cl{O}$, we may show $a'(p) \in \cl{O'}$; the converse direction is similar. There is an arrow $g:x_0 \to a(p)$. By $\cl{G}$-invariance of $a'$, both $g \cdot p_0$ and $p$ are in the same fiber of $a'$. Since the $\cl{H}$ action is transitive on this fiber, there is some arrow $h$ with $g \cdot p_0 \cdot h = p$. This action is only possible if $h$ is an arrow $h:a'(p) \to a'(g\cdot p_0) = a'(p_0)$. Thus $a'(p)$ and $a'(p_0)$ are in the same orbit $\cl{O'}$, as required.
\end{proof}
\begin{remark}
  \label{rem:8}
  The Hausdorff assumption is material, because if $\cl{G}$ and $\cl{H}$ are not Hausdorff, it is not clear that we may choose the invertible bibundle $P$ above to be Hausdorff. One possible attempt to treat the case of non-Hausdorff Lie groupoids is to modify the definition of Molino transverse equivalence to allow for non-Hausdorff manifolds. But doing so begs the question of how to define singular foliations on non-Hausdorff manifolds, which we leave for another time. Garmendia and Zambon encounter a similar problem in the case of Androulidakis-Skandalis singular foliations, and suggest it may be remedied by viewing Androulidakis-Skandalis singular foliations as sheaves instead of submodules, as in Remark \ref{rem:7}; see \cite[Section 4.1]{Garmendia2019}.
  \eor
\end{remark}

Given an arbitrary singular foliation $(N, \cl{F})$, it is unknown whether there is a Lie groupoid $\cl{G} \rra N$ which induces $\cl{F}$. Furthermore, it is possible that two non-Morita equivalent Lie groupoids induce the same singular foliation.
\begin{example}
  \label{ex:8}
  Consider the actions of the general linear group $\operatorname{GL}(n)$, and the special linear group $\operatorname{SL}(n)$, on $\R^n$. Both induce the same singular foliation of $\R^n$, whose leaves are the origin and its complement. However, the action groupoids $\operatorname{GL}(n) \ltimes \R^n$ and $\operatorname{SL}(n) \ltimes \R^n$ are not Morita equivalent. This is because their stabilizer groups at the origin are not isomorphic.
  \eoe
\end{example}

For regular foliations $(N, \cl{F})$, however, we have a distinguished groupoid inducing the foliation, the holonomy groupoid $\Hol(\cl{F})$ (see Example \ref{ex:13} for some details). It is reasonable to ask whether Molino transverse equivalence of regular foliations implies their holonomy groupoids are Morita equivalent, and indeed this is the case:

\begin{proposition}
  \label{prop:8}
  If two regular foliations are Molino transverse equivalent, then their holonomy groupoids are Morita equivalent.
\end{proposition}
This coincides with Proposition 3.30 in Garmendia and Zambon's paper \cite{Garmendia2019}, whose proof in the regular case is outlined after their Theorem 3.21. We recall that Garmendia and Zambon work with Androulidakis-Skandalis singular foliations, and not Stefan singular foliations. However, in the regular case, these notions coincide, so their proof of this statement works in this setting and we refer the reader there for details.

\begin{remark}
  \label{rem:9}
  In fact, Garmendia and Zambon prove a stronger statement. As we discussed in Section \ref{sec:andr-skand-sing}, Androulidakis-Skandalis singular foliations are not generally induced by any Lie groupoid. However, Androulidakis and Skandalis in \cite{Androulidakis2009} constructed an open topological groupoid for every Androulidakis-Skandalis singular foliation $\cl{F}$, called the \emph{holonomy groupoid} of $\cl{F}$, which coincides with the usual holonomy groupoid in the regular case. Garmendia and Zambon proved that, provided the holonomy groupoids are Hausdorff, they are Morita equivalent (as open topological groupoids) if and only if the associated singular foliations are what they call Hausdorff Morita equivalent.
  \eor
\end{remark}

\begin{corollary}
  \label{cor:3}
  Two regular foliations with Hausdorff holonomy groupoids are Molino transverse equivalent if and only if their holonomy groupoids are Morita equivalent.
\end{corollary}

We are now prepared to illustrate the example alluded to at the end of Section \ref{sec:molino-transv-equiv}, of two regular foliations with diffeomorphic leaf spaces, but which are not Molino transverse equivalent. This example also appears in \cite{Karshon2022}, in a different context.

\begin{example}
  \label{ex:9}
  Let $h:\R \to \R$ be a smooth non-negative function that is flat\footnote{this means that $h$ and all its derivatives vanish at the point} at $0$ and is positive everywhere else, such that the vector field $X:= h\pdd{}{x}$ is complete. Let $\psi := \Phi^X_1$ denote the time-1 flow of $X$, and set
  \begin{equation*}
    \hat{\psi}(x) :=
    \begin{cases}
      \psi(x) &\text{if } x \geq 0 \\
      \psi^{-1}(x) &\text{if } x < 0.
    \end{cases}
  \end{equation*}
  Both $\psi$ and $\hat{\psi}$ are smooth. By iterating $\psi$, we get a $\Z$-action on $\R$, and also one on $\R^2$ given by
  \begin{equation*}
    k \cdot (t,x) := (t+k, \psi^k(x)).
  \end{equation*}
  This action preserves the foliation of $\R^2$ by horizontal lines (but not the leaves themselves), and the action is free and properly discontinuous. Therefore, passing to the quotient, $M_\psi := \R^2/\psi$ is a manifold and we have the quotient foliation $\cl{F}_\psi$. Moreover, $M_\psi/\cl{F}_\psi$ is diffeomorphic to $\R/\psi$. Similarly, we can form $M_{\hat{\psi}}/\cl{F}_{\hat{\psi}}$, and this is diffeomorphic to $\R/\hat{\psi}$. But $\R / \psi = \R/\hat{\psi}$, so we have two foliations foliation with diffeologically diffeomorphic leaf spaces. However, by {\cite[Proposition 7.1]{Karshon2022}}, the \'{e}tale holonomy groupoids (see Example \ref{ex:13}) associated to these foliations are not Morita equivalent. Therefore their holonomy groupoids are not Morita equivalent, hence by Proposition \ref{prop:8}, the foliations are not transverse equivalent.
  \eoe
\end{example}
\subsection{Diffeology and Molino equivalence for quasifolds}
\label{sec:diff-molino-equiv}

While in general, a diffeological diffeomorphism between the leaf spaces of singular foliations does not induce a Molino transverse equivalence (c.f.\ Example \ref{ex:9}), there is a class of regular foliations for which this holds. These are regular foliations whose holonomy groupoids are Morita equivalent to quasifold groupoids, introduced in \cite{Karshon2022}.

\begin{definition}
  \label{def:12}
  A $n$-\emph{quasifold groupoid} is a Lie groupoid $\cl{G} \rra M$, with Hausdorff arrow space, such that: for each $x \in M$, there is an open neighbourhood $U$ of $x$, a countable group $\Gamma$ acting affinely on $\R^n$, an open subset $\mathsf{V}$ of $\R^n$, and an isomorphism\footnote{invertible functor with smooth inverse} of Lie groupoids $\cl{G}|_U \to (\Gamma \ltimes \R^n)|_{\mathsf{V}}$.
  \eod
\end{definition}
\begin{proposition}
  \label{prop:9}
  Assume $(N_i,\cl{F}_i)$ are regular foliations, and their holonomy groupoids $\Hol(\cl{F}_i)$ are each Morita equivalent to effective quasifold groupoids $\cl{G}_i \rra M_i$. If the leaf spaces $N_i/\cl{F}_i$ are diffeologically diffeomorphic, then the foliations $\cl{F}_i$ are Molino transverse equivalent.
\end{proposition}
\begin{proof}
  Morita equivalent Lie groupoids have diffeomorphic orbit spaces, therefore the $M_i/\cl{G}_i$ are diffeomorphic to $N_i/\cl{F}_i$, for each $i=0,1$. It follows from the assumption that the $M_i/\cl{G}_i$ are diffeomorphic to each other. Then, because the $\cl{G}_i$ are effective quasifold groupoids, we use {\cite[Proposition 5.4]{Karshon2022}} to conclude that the $\cl{G}_i$ are Morita equivalent. Therefore the $\Hol(\cl{F}_i)$ are Morita equivalent.

  If we show each $\Hol(\cl{F}_i)$ is Hausdorff, we may conclude the proof with Corollary \ref{cor:3}. But to be Hausdorff is invariant under Morita equivalence ({\cite[Proposition 5.3]{Moerdijk2003}}, and the $\cl{G}_i$ are Hausdorff by definition of quasifold groupoid, hence the holonomy groupoids are Hausdorff, as required.
\end{proof}

\begin{remark}
  \label{rem:10}
  The assumption that the $\cl{G}_i$ are effective is unnecessary. Each $\Hol(\cl{F}_i)$ is Morita equivalent to its \'{e}tale holonomy groupoid, as outlined in Example \ref{ex:13}, and these are effective. To be effective is stable under Morita equivalence ({\cite[Example 5.21 (2)]{Moerdijk2003}}, hence the $\cl{G}_i$ are necessarily effective.
  \eor
\end{remark}

In Proposition \ref{prop:9}, we impose a condition on the holonomy groupoid $\Hol(\cl{F}_i)$, and not on the foliation itself. There are cases, however, where we can deduce the holonomy groupoid is Morita equivalent to a quasifold groupoid directly from properties of the foliation. We end this brief section with an example.
\begin{example}
  \label{ex:10}
  A distinguished class of quasifold groupoids are proper \'{e}tale groupoids, which are sometimes called orbifold groupoids. Their local linear models arise as a consequence of the linearization theorem for proper groupoids, for example see \cite{Crainic2013}; the resulting groups $\Gamma$ are the stabilizer subgroups $\cl{G}_x$, which are finite. If all of the leaves of a regular foliation $\cl{F}$ are compact and have finite holonomy groups, by the Reeb stability theorem the holonomy groupoid $\Hol(\cl{F})$ is proper ({\cite[Example 5.28 (2)]{Moerdijk2003}}). Because to be proper is stable under Morita equivalence ({\cite[Proposition 5.26]{Moerdijk2003}}), we may apply Proposition \ref{prop:9} to such foliations $\cl{F}$.
  \eoe
\end{example}

\section{Basic cohomology of singular foliations}
\label{sec:basic-cohom-sing}

Given a singular foliation $(M, \cl{F})$, we have the associated complex of basic differential forms. In this section, we will show that Molino transverse equivalent singular foliations have identical complexes of basic forms, strengthening the notion that the basic complex captures the transverse geometry. We will end with a discussion of the comparison between the complex of basic differential forms, and the complex of diffeological differential forms on the leaf space.

\begin{definition}
  \label{def:13}
  A differential form $\alpha \in \cplx{}{M}$ is $\cl{F}$-\define{basic} if, for every vector field $X$ tangent to $\cl{F}$, both
  \begin{equation*}
    \iota_X\alpha = 0 \text{ and } \cl{L}_X\alpha = 0.
  \end{equation*}
  The first condition says $\alpha$ is \define{horizontal}. The second says $\alpha$ is \define{invariant}. We denote the collection of basic differential forms by $\cplx{b}{M,\cl{F}}$. This is a de-Rham subcomplex of $\cplx{}{M}$, with the usual differential. Its cohomology $H_b^\bullet(M, \cl{F})$ is the \define{basic cohomology} associated to $\cl{F}$.
  \eod
\end{definition}

\begin{example}
  \label{ex:11}
  If $G$ is a connected Lie group acting smoothly on $M$, and $\cl{F}$ is the associated singular foliation, the $\cl{F}$-basic forms are precisely those that are horizontal and $G$-invariant. More generally, if $\cl{G}$ is a source-connected Lie groupoid, and $\cl{F}$ is the associated singular foliation, the $\cl{F}$-basic forms are those for which $s^*\alpha = t^*\alpha$, see {\cite[Proposition 5.5]{Miyamoto2022}}.
  \eoe
\end{example}

\begin{proposition}
  \label{prop:10}
 Suppose $p:M \to N$ is a surjective submersion with connected fibers, and that $(N, \cl{F})$ is a singular foliation. Then $p^*$ is an isomorphism from $\cplx{b}{N, \cl{F}}$ to $\cplx{b}{M, p^{-1}\cl{F}}$.
\end{proposition}
\begin{proof}
  Because $p$ is a submersion, $p^*$ is injective. It remains to show $p^*$ is into, and onto. First, let $\alpha$ be an $\cl{F}$-basic form on $N$, and take $X$ tangent to $p^{-1}\cl{F}$. Then
  \begin{equation*}
    (\iota_X p^*\alpha)_x = \alpha_{p(x)}(p_*X_x,\cdot) = (\iota_{p_*X_x}\alpha)_x = 0,
  \end{equation*}
  since $p_*X_x$ is tangent to $\cl{F}$. As $d\alpha$ is also $\cl{F}$-basic, by a similar computation we have $\iota_X(p^*d\alpha) = 0$, so by Cartan's formula,
  \begin{equation*}
    \cl{L}_X p^*\alpha = \iota_X(dp^*\alpha) = \iota_X(p^*d\alpha) = 0.
  \end{equation*}
  
  Second, let $\beta$ be a $p^{-1}\cl{F}$-basic form on $M$. Since $p$ is a surjective submersion with connected fibers, its fibers form a regular foliation $\cl{F}'$ of $M$. Any vector field tangent to these fibers is also tangent to $p^{-1}\cl{F}$, so the $p^{-1}\cl{F}$-basic forms are all $\cl{F}'$-basic. In particular, $\beta$ is $\cl{F}'$-basic. Now, we can identify $N \cong M/\cl{F}'$, and we can identify the quotient with $p:M \to N$. By {\cite[Exercise 14.9]{Lee2013}}, the pullback $p^*$ is an isomorphism from $\cplx{}{N}$ to $\cplx{b}{M, \cl{F}'}$. This provides $\alpha \in \cplx{}{N}$ such that $p^*\alpha = \beta$.

  We now check that $\alpha$ is $\cl{F}$-basic. Take $Y$ tangent to $\cl{F}$, and fix $x \in M$. By Lemma \ref{lem:2}, lift $Y$ to a $p$-related vector field $X$ about $x$, which is tangent to $p^{-1}\cl{F}$. Since $\beta$ is horizontal,
  \begin{equation*}
    0 = (p^*\alpha)_x(X,\cdot) = \alpha_{p(x)}(Y_{p(x)}, p_*\cdot).
  \end{equation*}
  But $p$ and $p_*$ are onto, so we conclude $\iota_Y\alpha = 0$, and $\alpha$ is horizontal. We also check that since $\beta$ is invariant,
  \begin{equation*}
    p^*(\cl{L}_Y\alpha) = L_Xp^*\alpha = 0.
  \end{equation*}
  But $p^*$ is injective, so $\cl{L}_Y\alpha = 0$, and $\alpha$ is invariant. This completes the proof.
\end{proof}
\begin{corollary}
  \label{cor:4}
  If two singular foliations are Molino transverse equivalent, then their complexes of basic forms, hence their basic cohomologies, are isomorphic. 
\end{corollary}

Diffeology provides an alternative approach to using differential forms to study the transverse geometry of a singular foliation. Every diffeological space $X$ comes with a de Rham complex of diffeological differential forms, $\cplx{}{X}$. When $X$ is a leaf space of a singular foliation $(N, \cl{F})$, we can directly compare $\cplx{}{N/\cl{F}}$ and $\cplx{b}{N, \cl{F}}$.

\begin{theorem}
  \label{thm:6}
  Fix a singular foliation $(N, \cl{F})$. The quotient map $\pi:N \to N/\cl{F}$ induces a map of chain complexes $\pi^*:\cplx{}{N/\cl{F}} \to \cplx{}{N}$. This is injective, and its image is contained in the space of $\cl{F}$-basic forms. The pullback $\pi^*$ is an isomorphism $\cplx{}{N/\cl{F}} \to \cplx{b}{N, \cl{F}}$ whenever
  \begin{itemize}
  \item the set of points in leaves of dimension $k$ is a diffeological submanifold of $N$, for each $k$, or
    \item $\pi^*$ is an isomorphism when we replace $(N, \cl{F})$ with $N_{>0}$ ($N$ with the zero-leaves excised) and the induced singular foliation $\cl{F}_{>0}$.
  \end{itemize}
\end{theorem}
This is the main result of \cite{Miyamoto2022}. It implies the earlier results from Karshon and Watts \cite{Karshon2016} and Watts \cite{Watts2022}, in the case of connected groups and source-connected groupoids, respectively. For regular foliations, this was proved by Hector, Marc\'{i}as-Virg\'{o}s, and Sanmart\'{i}n-Carb\'{o}n in \cite{Hector2011}. The following question remains open:
\begin{question}
  Is $\pi^*:\cplx{}{N/\cl{F}} \to \cplx{b}{N, \cl{F}}$ always an isomorphism?
\end{question}
However, Corollary \ref{cor:4} lets us show $\pi^*$ is an isomorphism for an \emph{a priori} larger class of singular foliations than Theorem \ref{thm:6}.

\begin{proposition}
  \label{prop:11}
  Suppose $(N_i, \cl{F}_i)$ are Molino transverse equivalent singular foliations, and assume that $\pi_i:N_i \to N_i/\cl{F}_i$ induces an isomorphism $\pi_i^*:\cplx{}{N_i/\cl{F}_i} \to \cplx{b}{N_i, \cl{F}_i}$ for $i=0$. Then the same holds for $i=1$.
\end{proposition}
\begin{proof}
  Using Proposition \ref{prop:6}, create the following commutative pentagon.
  \begin{equation*}
    \begin{tikzcd}
      & M \ar[dl, "p_0"'] \ar[dr, "p_1"] & \\
      N_0 \ar[d, "\pi_0"'] & & N_1 \ar[d, "\pi_1"] \\
      N_0/\cl{F}_0 \ar[rr, "\varphi"] & & N_1/\cl{F}_1.
    \end{tikzcd}
  \end{equation*}
  All of $\varphi^*$, $\pi_0^*$, $p_0^*$, and $p_1^*$ are isomorphisms of the relevant complexes. Thus $\pi_1^*$ is too.
\end{proof}

We could also make the same statement in cohomology.
\appendix

\section{Review of Lie groupoids}
\label{sec:review-lie-groupoids}

Here we establish some necessary facts and terminology about Lie groupoids. For details, see \cite{Moerdijk2003} or \cite{Lerman2010}. A \define{Lie groupoid} is a category $\cl{G} \rra M$, whose arrows are all invertible, such that: $M$, called the \define{base}, is a manifold, $\cl{G}$ has a smooth structure (but is not necessarily Hausdorff or second-countable), all the structure maps are smooth, and the source $s$ (hence the target $t$) is a submersion with Hausdorff fibers. We typically denote an arrow $g$ with source $x$ and target $x'$ by $g:x \to x'$. The relation $x \sim x'$ if there is an arrow $g:X \to x'$ is an equivalence relation, whose classes we call \define{orbits}.

\begin{example}
  \label{ex:12}
  If $G$ is a Lie group acting smoothly on a manifold $M$, we can form the \define{action groupoid} $G \ltimes M$, with arrow space $G \times M$, base space $M$, source $(g,x) \mapsto x$, target $(g,x) \mapsto g \cdot x$, and multiplication $(g',x')(g,x) = (g'g,x)$. In this way, Lie groupoids subsume Lie group actions.
  \eoe
\end{example}

If the fibers of the source map $s$ are connected, we say $\cl{G}$ is \define{source-connected}. When the arrow space $\cl{G}$ is Hausdorff, we call $\cl{G}$ itself Hausdorff. If $\cl{G}$ is Hausdorff, and the map $(s,t):\cl{G} \to M \times M$ is proper (i.e. pre-images of compact sets are compact), we call $\cl{G}$ \define{proper}.

Given an immersed submanifold $\iota: (N, \sigma_N) \to M$, if $(t,s):\cl{G} \to M \times M$ is transverse to $(\iota, \iota): N \to M\times M$, then we can form the \define{pullback} of $\cl{G}$ along $\iota$, and obtain $\iota^*\cl{G} \rra N$, whose arrows $y \to y'$ are those arrows in $\cl{G}$ from $\iota(y) \to \iota(y')$. If $U$ is an open subset of $M$, we denote the pullback of $\cl{G}$ to $U$ by $\cl{G}|_U \rra U$. Its arrows are the arrows of $\cl{G}$ with source and target in $U$.

We call a Lie groupoid \define{\'{e}tale} if its arrow space and base space have the same dimension. In this case, the source and target are local diffeomorphisms. If the map
\begin{equation*}
  g \mapsto \operatorname{germ}_{s(g)}(t \circ s^{-1}),
\end{equation*}
where $s^{-1}$ is the local inverse of $s$ mapping $s(g)$ to $g$, is injective, we call the \'{e}tale groupoid $\cl{G}$ \define{effective}.

\begin{definition}
  \label{def:14}
  A \define{left action} of a Lie groupoid $\cl{G} \rra M$ on a manifold $P$ consists of a \define{multiplication} $\mu:\cl{G} \fiber{s}{a} P \to P$, which we usually denote by $\cdot$, and an \define{anchor} $a:P \to M$, such that the following diagram commutes:
  \begin{equation*}
  \begin{tikzcd}
      \cl{G} \fiber{s}{a} P \ar[r, "\mu"] \ar[d, "\pr_1"] & P \ar[d, "a"] \\
      \cl{G} \ar[r, "t"] & M,
    \end{tikzcd}
  \end{equation*}
  and furthermore
  \begin{itemize}
  \item $g' \cdot (g \cdot p) = g'g \cdot p$ whenever this makes sense, and
    \item $1_{a(p)}\cdot p = p$ for all $p \in P$.
    \end{itemize}
    \eod
\end{definition}
We can similarly define a right action.
\begin{definition}
  \label{def:15}
  An \define{invertible bibundle} from $\cl{G}$ to $\cl{H}$, denoted $P:\cl{G} \to \cl{H}$, is a (not-necessarily Hausdorff) manifold $P$ on which $\cl{G}$ acts from the left (along anchor $a$) and $\cl{H}$ acts on the right (along anchor $a'$), such that the anchors are surjective submersions, the actions commute, and
  \begin{itemize}
  \item $a$ is $\cl{H}$-invariant, and $a'$ is $\cl{G}$-invariant, and
  \item the maps
    \begin{equation*}
      \cl{G} \fiber{s}{a} P \to P \fiber{a'}{a'} P, \quad (g,p) \mapsto (g\cdot p, p),
    \end{equation*}
    and
    \begin{equation*}
      P \fiber{a'}{t} \cl{H} \to P \fiber{a}{a} P, \quad (p,h) \mapsto (p, p\cdot h)
    \end{equation*}
    are diffeomorphisms. Equivalently, $\cl{G}$ and $\cl{H}$ act freely and transitively on the fibers of $a'$ and $a$, respectively
  \end{itemize}

  We say $\cl{G}$ and $\cl{H}$ are \define{Morita equivalent} if there is an invertible bibundle between them.
 \eod
\end{definition}

\begin{lemma}
  \label{lem:3}
  If two Hausdorff Lie groupoids are Morita equivalent, there is a Hausdorff invertible bibundle between them.
\end{lemma}
For a proof, see {\cite[Corollary A.7]{Garmendia2019}}.

\begin{example}
  \label{ex:13}
  Given a regular foliation $(M, \cl{F})$, we can construct the \define{holonomy groupoid} $\Hol(\cl{F}) \rra M$. There is an arrow $x \to y$ if and only if $x$ and $y$ are in the same leaf $L$; in this case, each arrow is a \emph{holonomy class} of a path from $x$ to $y$ within $L$. Two paths are in the same holonomy class if they induce the same diffeomorphism between small transversal slices (submanifolds which intersect the leaves transversely)  at $x$ and $y$. Paths in the same homotopy class are in the same holonomy class.

  The holonomy groupoid is source-connected. Given a complete transversal $T$ to $\cl{F}$, i.e.\ an immersed submanifold $(T, \sigma_T)$ which intersects each leaf of $\cl{F}$ at least once, and always intersects leaves transversely, we can pull back the holonomy groupoid to $T$, obtaining $\Hol(\cl{F})|_T \rra T$. This is an \'{e}tale groupoid. Different choices of complete transversals $T$ yield Morita equivalent pullback groupoids, hence we call $\Hol(\cl{F})|_T \rra T$ ``the'' \emph{\'{e}tale holonomy groupoid} of $\cl{F}$. This groupoid is always effective.
  \eoe
\end{example}

\begin{definition}
  A \define{Lie Algebroid} is a vector bundle $A \to M$, together with a Lie bracket on the space of sections of $A$, and a bundle map $a:A \to TM$, the \define{anchor}, satisfying a Leibniz rule with respect to the algebroid bracket. We denote a Lie algebroid by $A \implies M$.
  \eod
\end{definition}
To each Lie groupoid $\cl{G} \rra M$, we can associate a Lie algebroid $\operatorname{Lie}(\cl{G}) \to M$. The maximal integral submanifolds of the distribution $a(\operatorname{Lie}(\cl{G}))$ are the connected components of the orbits of $\cl{G}$.

\mute{}{
\newpage

\section{Extra things}
\label{sec:extra-things}

\begin{definition}
  Given a topological space $X$, a \define{maximal Hausdorff quotient} $\hat{X}$ is a Hausdorff space together with a quotient map $\pi:X \to \hat{X}$ satisfying the following universal property: for any Hausdorff space $Y$, any continuous function $f:X \to Y$ factors uniquely through $\pi$:
  \begin{equation*}
    \begin{tikzcd}
      X \ar[d, "\pi"'] \ar[dr, "f"] & \\
      \hat{X} \ar[r, dashed, "\exists ! \hat{f}"'] & Y.
    \end{tikzcd}
  \end{equation*}
  Maximal Hausdorff quotients always exist, and are unique up to isomorphism.

  \begin{question}
    If $X$ is a non-Hausdorff smooth manifold, is $\hat{X}$ necessarily a Hausdorff smooth manifold?
  \end{question}
\end{definition}

\begin{lemma}
  Suppose $P: \cl{G} \to \cl{H}$ represents a Morita equivalence between source-connected Lie groupoids $\cl{G}$ and $\cl{H}$. Passing to the maximal Hausdorff quotient gives the following diagram:
  \begin{equation*}
    \begin{tikzcd}
      \cl{G} \circlearrowright & P \ar[ddl, "a"'] \ar[ddr, "a'"] \ar[d, "\pi"] & \circlearrowleft \cl{H} \\
      & \hat{P} \ar[dl, "\hat{a}"] \ar[dr, "\hat{a'}"'] & \\
      M & & N.
    \end{tikzcd}
  \end{equation*}
  If $\hat{P}$ is a smooth manifold (meaning $\pi$ is a submersion), then it, together with $\hat{a}$ and $\hat{a'}$, gives a transverse Molino equivalence between the singular foliations induced by $\cl{G}$ and $\cl{H}$.
\end{lemma}
\begin{proof}
  We check that $\hat{a}$ is a surjective submersion with connected fibers, and that $\hat{a}^{-1}\cl{F}_{\cl{G}} = \hat{a'}^{-1}\cl{F}_{\cl{H}}$. Because both $a$ and $\pi$ are surjective submersions, commutativity implies that so is $\hat{a}$. For the last two conditions, observe that in terms of images and pre-images,
  \begin{equation*}
   \hat{a}^{-1} = \pi \circ \pi^{-1} \circ \hat{a}^{-1} = \pi \circ a^{-1}, 
  \end{equation*}
  where we use commutativity and the fact $\pi$ is surjective. This immediately shows that $\hat{a}$ has connected fibers (because $a$ does), and, using a similar equality for $\hat{a'}$, that if two orbits correspond under $a$ and $a'$, they also correspond under $\hat{a}$ and $\hat{a'}$.
\end{proof}
}
\end{document}